\documentclass[a4paper,leqno,12pt,draft]{article}

\usepackage{amsmath}
\usepackage{amssymb}
\usepackage{amsfonts}
\usepackage{amsthm}
\usepackage{enumerate}
\usepackage{comment}
\usepackage{cite}
\usepackage[dvips]{graphicx,color}
\usepackage{delarray}
\usepackage{ascmac}
\usepackage{fancybox}

\theoremstyle{plain}
\newtheorem{theorem}{Theorem}[section]

\newtheorem{proposition}[theorem]{Proposition}
\newtheorem{lemma}[theorem]{Lemma}
\newtheorem{remark}[theorem]{Remark}

\numberwithin{equation}{section}

\begin{document}

\title{\it {\Large Stability analysis of an overdetermined fourth order boundary value problem via an integral identity}}
\author{Yuya Okamoto \and Michiaki Onodera}
\date{}
\maketitle

\begin{abstract}
We consider an overdetermined fourth order boundary value problem in which the boundary value of the Laplacian of the solution is prescribed, in addition to the homogeneous Dirichlet boundary condition. 
It is known that, in the case where the prescribed boundary value is a constant, this overdetermined problem has a solution if and only if the domain under consideration is a ball. 
In this paper, we study the shape of a domain admitting a solution to the overdetermined problem when the prescribed boundary value is slightly perturbed from a constant. 
We derive an integral identity for the fourth order Dirichlet problem and a nonlinear weighted trace inequality, and the combination of them results in a quantitative stability estimate which measures the deviation of a domain from a ball in terms of the perturbation of the boundary value. 
\end{abstract}

\section{Introduction}
\label{section-introduction}

We consider the fourth order Dirichlet problem
\begin{equation}
\label{dirichlet}
\left\{
\begin{aligned}
 \Delta^2 u&=1 && \textrm{in} \ \Omega,\\
 u=\frac{\partial u}{\partial \nu}&=0 && \textrm{on} \ \partial\Omega
\end{aligned}
\right.
\end{equation}
with the additional overdetermined boundary condition
\begin{equation}
\label{od}
 \Delta u=f \quad \textrm{on} \ \partial\Omega,
\end{equation}
where $\Omega\subset\mathbb{R}^n$ ($n\geq 2$) is a bounded domain, $\nu$ is the unit outer normal vector to $\partial\Omega$ and $f$ is a prescribed positive function defined in $\mathbb{R}^n$. 
Equation \eqref{dirichlet} models the bending of a horizontally clamped plate being pushed upward and properties of the solution $u$ are extensively studied in Gazzola, Grunau and Sweers \cite{gazzola_grunau_sweers-2010}. 
Our interest lies in the shape of $\Omega$ for which a unique solution $u$ to \eqref{dirichlet} additionally satisfies \eqref{od}. 


The corresponding second order overdetermined problem is
\begin{equation}
\label{torsion}
 \quad \left\{
 \begin{aligned}
  -\Delta \psi&=1 && \text{in} \ \Omega,\\
  \psi&=0 && \text{on} \ \partial\Omega
 \end{aligned}
 \right.
\end{equation}
with the additional boundary condition
\begin{equation}
\label{od-2}
 -\frac{\partial\psi}{\partial\nu}=|\nabla \psi|=f \quad \text{on} \ \partial\Omega. 
\end{equation}
A celebrated result of Serrin \cite{serrin-1971} states that, in the case where $f$ is a constant, \eqref{torsion} together with \eqref{od-2} has a solution $\psi$ if and only if $\Omega$ is a ball. 
In the poof he introduced the so-called method of moving planes, based on Alexandrov's reflection principle and an extension of Hopf's lemma applicable to boundary points at corners, and proved the symmetry result for more general second order overdetermined problems. 
Weinberger \cite{weinberger-1971} provided a simple alternative proof based on the observation that, for a solution $\psi$ to \eqref{torsion}, the nonnegative function called the Cachy-Schwarz deficit
\begin{equation*}
 \delta(\psi):=|D^2\psi|^2-\frac{(\Delta \psi)^2}{n}=\sum_{i,j=1}^n(\partial_{ij}\psi)^2-\frac1n\left(\sum_{k=1}^n\partial_{kk}\psi\right)^2
\end{equation*}
vanishes everywhere if and only if $\psi$ is a quadratic polynomial of the form
\begin{equation}
\label{quadratic_polynomial}
 Q(x)=\frac{R^2-|x-z|^2}{2n} \quad (R>0, \ z\in\mathbb{R}^n), 
\end{equation}
which is an explicit solution to \eqref{torsion} and \eqref{od-2} when $\Omega$ is a ball and $f$ is a constant. 
This fact suggests that $\delta(\psi)$ measures the deviation of $\psi$ from $Q$, or that of $\Omega$ from a ball. 
This observation was further extended by Magnanini and Poggesi \cite{magnanini_poggesi-2019,magnanini_poggesi-2020_indiana,magnanini_poggesi-2020}, and they showed that the integral identity
\begin{equation}
\label{integral_identity_2nd}
 \int_\Omega \psi\left\{\left|D^2 \psi\right|^2-\frac{(\Delta \psi)^2}{n}\right\}\,dx=\frac{1}{2}\int_{\partial\Omega}\left(c^2-|\nabla\psi|^2\right)\left\{\frac{\partial\psi}{\partial\nu}+\frac{(x-z)\cdot\nu}{n}\right\}\,dS
\end{equation}
holds for any solution $\psi$ to \eqref{torsion}, $z\in\Omega$ and $c\in\mathbb{R}$. 
This identity directly implies the radial symmetry of $\Omega$ when \eqref{od-2} is satisfied for a constant $f=c$, since the right hand side of the identity then becomes $0$ and thus $\delta(\psi)=0$. 
Moreover, by estimating the both sides of the identity by elaborate inequalities for harmonic functions, in particular for $\psi-Q$, they obtained the stability estimate
\begin{equation}
\label{sharp_estimate}
 \rho_2-\rho_1\leq C\left\||\nabla\psi|-c\right\|_{L^2(\partial\Omega)}^{\tau_n}
\end{equation}
for the shape of an unknown domain $\Omega$ in terms of the radii $0<\rho_1\leq \rho_2<\infty$ of the largest ball contained in $\Omega$ and the smallest ball containing $\Omega$ with common center $z\in\mathbb{R}^n$, i.e., $B_{\rho_1}(z)\subset\Omega\subset B_{\rho_2}(z)$, where $\tau_2=1$, $\tau_3$ is arbitrarily close to $1$, and $\tau_n=2/(n-1)$ for $n\geq 4$. 
This shows quantitatively how $\Omega$ is close to a ball in the Hausdorff distance when the additional condition \eqref{od-2} is satisfied for $f$ close to a constant in the $L^2$-norm. 

These types of stability estimates for the second order problem \eqref{torsion} and \eqref{od-2} had been obtained also by means of a quantitative version of the method of moving planes, initiated by Aftalion, Busca and Reichel \cite{aftalion_busca_reichel-1999}, and developed by Ciraolo, Magnanini and Vespri \cite{ciraolo_magnanini_vespri-2016} for some $0<\tau_n<1$ with the $L^2$-norm replaced by the Lipschitz seminorm. 
In fact, these results also hold for semilinear equations $-\Delta u=f(u)$ with $u>0$. 
On the other hand, Brandolini, Nitsch, Salani and Trombetti \cite{brandolini_nitsch_salani_trombetti-2008_2} made use of an integral quantity related to Newton's inequality involving elementary symmetric functions of the eigenvalues of $D^2 \psi$, and proved the same stability result for some $0<\tau_n<1$ with the $L^\infty$-norm in the right hand side. 
Moreover, this argument was used to yield an estimate of the volume of the symmetric difference of $\Omega$ and a union of balls by the $L^1$-norm of $|\nabla\psi|-c$. 
Developing their idea, Feldman \cite{feldman-2018} proved the linear stability estimate (i.e., $\tau_n=1$) for the volume $|\Omega\triangle B|$ of the symmetric difference of $\Omega$ and a ball $B$ by the $L^2$-norm of $|\nabla\psi|-c$. 
Recently, Gilsbach and Onodera \cite{gilsbach_onodera-2021} installed a new implicit function theorem and applied it to establish a linear optimal stability estimate with H\"older norms in both sides of the inequality. 

For higher order overdetermined problems, the radial symmetry of $\Omega$ for constant boundary values were studied by Bennett \cite{bennett-1986}, Payne and Schaefer \cite{payne_schaefer-1989,payne_schaefer-1994}, Dalmasso \cite{dalmasso-1990}, Philippin and Ragoub \cite{philippin_ragoub-1995}, Barkatou \cite{barkatou-2008} and Colasuonno and Vecchi \cite{colasuonno_vecchi-2019}. 
In particular, for the fourth order overdetermined problem \eqref{dirichlet} and \eqref{od} with $f=c$, Bennett \cite{bennett-1986} proved the radial symmetry of $\Omega$ by extending Weinberger's argument and deriving $\delta(v)=0$, where $v:=-\Delta u$ is a solution to
\begin{equation*}
 \left\{
 \begin{aligned}
  -\Delta v&= 1 && \text{in} \ \Omega,\\
  v&=-f && \text{on} \ \partial\Omega. 
 \end{aligned}
 \right.
\end{equation*}
Note that, when $\Omega$ is a ball, the overdetermined problem \eqref{dirichlet} with \eqref{od} has a radially symmetric solution $u$ given by a quartic polynomial (see \eqref{explicit_radial_solution}), and $v$ is a quadratic polynomial of the form \eqref{quadratic_polynomial}. 
Thus $\delta(v)$ is expected to measure the deviation of $\Omega$ from a ball. 
It is our attempt to derive an integral identity, analogous to \eqref{integral_identity_2nd}, involving $\delta(v)$ and the deviation $\Delta u-c$ of the additional boundary value in \eqref{od} from a constant. 
The following integral identity is one of the key ingredients of our stability analysis. 

\begin{theorem}
\label{theorem-identity}
Let $\Omega\subset\mathbb{R}^n$ be a bounded domain having $C^1$-boundary $\partial\Omega$ and $u\in C^{4}(\overline{\Omega})$ a solution to \eqref{dirichlet}, and set $v:=-\Delta u$. 
Then, 
\begin{equation}
\label{integral_identity_4th}
 \int_\Omega u\left\{|D^2 v|^2-\frac{(\Delta v)^2}{n}\right\}\,dx=\frac14\int_{\partial\Omega}\left\{c^2-(\Delta u)^2\right\}\left\{\frac{\partial v}{\partial \nu}+\frac{(x-z)\cdot\nu}{n}\right\}\,dS
\end{equation}
holds for any $z\in\Omega$ and $c\in\mathbb{R}$. 
\end{theorem}

An immediate consequence of this identity is the radial symmetry of $\Omega$ when \eqref{dirichlet} and \eqref{od} with $f=c$ have a solution $u$, if additionally $\Omega$ is assumed to be $\varepsilon_0$-close to the unit ball $\mathbb{B}$ in the $C^{4+\alpha}$-sense ($0<\alpha<1$), i.e., there is a diffeomorphism $\Phi\in C^{4+\alpha}(\overline{\mathbb{B}},\overline{\Omega})$ with $\|\Phi-{\rm Id}\|_{C^{4+\alpha}(\overline{\mathbb{B}})}<\varepsilon_0$ for some particular constant $\varepsilon_0=\varepsilon_0(n)>0$ depending only on the dimension $n$. 
This additional requirement is due to the fact that the positivity preserving property, i.e., $\Delta^2 u\geq0$ implying $u\geq0$ in $\Omega$, for the fourth order Dirichlet problem \eqref{dirichlet} is no longer true for general bounded domains $\Omega$ including even mildly eccentric ellipses in $\mathbb{R}^2$. 

Stability estimates as \eqref{sharp_estimate} for higher order overdetermined problems have not been studied in the literature. 
This is particularly due to the failure of maximum principles for higher order equations. 
Recently we learned that Gilsbach and Stollenwerk \cite{gilsbach_stollenwerk-2021} obtained a stability estimate with H\"older norms by the implicit function theorem introduced in \cite{gilsbach_onodera-2021}. 
Our main result in this paper is the following quantitative stability estimate of the deviation of $\Omega$ from $\mathbb{B}$ in the Hausdorff distance by the $L^p$-norm of a perturbation of $f$ from a constant in \eqref{od}, under the $\varepsilon_0$-closeness of $\Omega$ to $\mathbb{B}$ in the $C^{4+\alpha}$-sense, where
\begin{equation*}
 \sigma_p:=\left\{
 \begin{aligned}
  & \frac{(n+2)p}{n(n+2p-1)} && (1\leq p<3/2),\\
  & \hspace{10mm} \frac{3}{2n} && (3/2\leq p\leq\infty). 
 \end{aligned}
 \right.
\end{equation*}

\begin{theorem}
\label{theorem-stability}
For $1\leq p\leq\infty$ and $0\leq\sigma<\sigma_p$, there are $\varepsilon_0>0$ and $C>0$ such that, if $\Omega$ is $\varepsilon_0$-close to $\mathbb{B}$ in the $C^{4+\alpha}$-sense and $u\in C^{4+\alpha}(\overline{\Omega})$ is a solution to \eqref{dirichlet}, then there are $z\in\Omega$ and $0<\rho_1\leq\rho_2<\infty$ such that $B_{\rho_1}(z)\subset\Omega\subset B_{\rho_2}(z)$ and
\begin{equation}
\label{stability_estimate}
 \rho_2-\rho_1\leq C\left(\left\|\Delta u-c\right\|_{L^p(\partial\Omega)}^{\sigma}+\left\|\Delta u-c\right\|_{L^\infty(\partial\Omega)}\right). 
\end{equation}
In the case where $n\geq 3$ and $p=1$, we can choose $\sigma=\sigma_1$, i.e., 
\begin{equation*}
 \rho_2-\rho_1\leq C\left(\left\|\Delta u-c\right\|_{L^1(\partial\Omega)}^\frac{n+2}{n(n+1)}+\left\|\Delta u-c\right\|_{L^\infty(\partial\Omega)}\right). 
\end{equation*}
\end{theorem}

\begin{remark}
{\rm 
The $\varepsilon_0$-closeness is used essentially only for the positivity preserving property for the fourth order Dirichlet problem \eqref{dirichlet} (see Proposition \ref{proposition-distance}). 
Although this assumption will be used in several other arguments such as the uniform Schauder estimates (Lemmas \ref{lemma-schauder} and \ref{lemma-perturbation}) and weighted inequalities in Section \ref{section-weighted_estimates}, these estimates hold for more general domains (see e.g.\ the proof of \cite[Lemma 2.7]{magnanini_poggesi-2020}) under some mild geometric conditions. 
}
\end{remark}

The integral-identity strategy used in this paper has also been performed in several related problems. 
For recent advances, see Dipierro, Poggesi and Valdinoci \cite{dipierro_poggesi_valdinoci-2021}, Fogagnolo and Pinamonti \cite{fogagnolo_pinamonti-2021}, and Scheuer \cite{scheuer-2021}. 
In fact, Magnanini and Poggesi \cite{magnanini_poggesi-2019} initiated this approach first for the stability analysis of Alexandrov's soap bubble theorem, which has a link with the second order problem \eqref{torsion} and \eqref{od-2} via integral identities. 
It would be very interesting to find such a relation between a forth order problem and its counterpart in differential geometry. 

This paper is organized as follows. 
In Section \ref{section-integral_identity}, we prove the integral identity \eqref{integral_identity_4th} for the fourth order Dirichlet problem \eqref{dirichlet} by virtue of a higher order analogue of Pohozaev's identity known as the Pucci-Serrin identity. 
The proof of Theorem \ref{theorem-stability} is based on estimating both sides of \eqref{integral_identity_4th} by several weighted inequalities for the harmonic function $h=v-Q$ and its gradients. 
For this strategy, in Section \ref{section-uniform_estimates}, we prove several uniform estimates including a pointwise estimate showing that $u$ behaves like the square of the distance $d_{\partial\Omega}$ from $\partial\Omega$ near $\partial\Omega$. 
In Section \ref{section-weighted_estimates}, we derive a new nonlinear weighted trace inequality of the form
\begin{equation*}
 \left\|\nabla h\right\|_{L^p(\partial\Omega)}\leq C\left\|d_{\partial\Omega}D^2h\right\|_{L^2(\Omega)}^{2\beta}
\end{equation*}
for some $0<2\beta<1$ specified explicitly in Lemma \ref{lemma-nonlinear_trace_inequality}. 
This together with several known inequalities deduces the stability estimate \eqref{stability_estimate}. 

\section{Integral identity for the Dirichlet problem}
\label{section-integral_identity}

The following lemma is a special case of the Pucci-Serrin variational identity \cite{pucci_serrin-1986} for the biharmonic operator. 

\begin{lemma}
\label{lemma-variational_identity}
Let $\Omega\subset\mathbb{R}^n$ be a bounded domain having $C^1$-boundary $\partial\Omega$ and suppose that $u\in C^{4}(\overline{\Omega})$ is a solution to \eqref{dirichlet}. 
Then, for any $z\in\Omega$, 
\begin{equation*}
(n+4)\int_\Omega u\,dx=\int_{\partial\Omega}(\Delta u)^2(x-z)\cdot\nu\,dS. 
\end{equation*}
\end{lemma}

\begin{proof}
For the reader's convenience, we give a proof for this special case. 
Let us set $w:=(x-z)\cdot\nabla u$ and observe that $u=|\nabla u|=w=0$ on $\partial\Omega$ and
\begin{align*}
 \Delta^2 w&=\Delta\left\{(x-z)\cdot\nabla\Delta u+2\Delta u\right\}=(x-z)\cdot\nabla\Delta^2 u+4\Delta^2 u=4 && \text{in} \ \Omega,\\
 \frac{\partial w}{\partial\nu}&={}^t(x-z)\cdot D^2u\cdot \nu=\left((x-z)\cdot\nu\right)\frac{\partial^2 u}{\partial\nu^2}=\left((x-z)\cdot\nu\right)\Delta u && \text{on} \ \partial\Omega. 
\end{align*}
Hence, by Green's identity
\begin{equation*}
 \int_\Omega\left(u\Delta^2w-\Delta^2u w\right)\,dx
 =\int_{\partial\Omega}\left(u\frac{\partial \Delta w}{\partial\nu}-\frac{\partial u}{\partial\nu}\Delta w+\Delta u\frac{\partial w}{\partial\nu}-\frac{\partial\Delta u}{\partial\nu}w\right)\,dS, 
\end{equation*}
we obtain 
\begin{equation*}
 \int_\Omega\left\{4u-(x-z)\cdot\nabla u\right\}\,dx=\int_{\partial\Omega}(\Delta u)^2(x-z)\cdot\nu\,dS, 
\end{equation*}
where the left hand side is equal to $(n+4)\int_\Omega u\,dx$ by the divergence theorem. 
\end{proof}

\begin{proof}[Proof of Theorem \ref{theorem-identity}]
Let us introduce the auxiliary function
\begin{equation*}
 q:=\frac{v^2}{4}-\frac{n+2}{2n}u, 
\end{equation*}
which satisfies
\begin{align*}
 \Delta q&=-\frac{v}{2}+\frac{|\nabla v|^2}{2}+\frac{n+2}{2n}v=\frac{|\nabla v|^2}{2}+\frac{v}{n},\\
 \Delta^2q&=|D^2v|^2-\frac1n=|D^2v|^2-\frac{(\Delta v)^2}{n}. 
\end{align*}
Applying Green's identity and using the boundary conditions, we have
\begin{align*}
 \int_\Omega u\Delta^2q\,dx
 &=\int_\Omega \Delta^2 u q\,dx+\int_{\partial\Omega}\left(\Delta u\frac{\partial q}{\partial\nu}-\frac{\partial\Delta u}{\partial\nu} q\right)\,dS\\
 &=\int_\Omega\left(\frac{v^2}{4}-\frac{n+2}{2n}u\right)\,dx+\int_{\partial\Omega}\left(-\frac{1}{2}v^2\frac{\partial v}{\partial\nu}+\frac{1}{4}v^2\frac{\partial v}{\partial\nu}\right)\,dS\\
 &=-\frac{n+4}{4n}\int_\Omega u\,dx-\frac{1}{4}\int_{\partial\Omega}v^2\frac{\partial v}{\partial\nu}\,dS, 
\end{align*}
where the last equality follows from
\begin{equation*}
 \int_\Omega v^2\,dx=-\int_{\Omega}v\Delta u\,dx=\int_{\Omega}u\,dx. 
\end{equation*}
Hence, by Lemma \ref{lemma-variational_identity}, 
\begin{align*}
 \int_{\Omega}u\Delta^2 q\,dx&=-\frac{1}{4n}\int_{\partial\Omega}(\Delta u)^2(x-z)\cdot\nu\,dS-\frac{1}{4}\int_{\partial\Omega}(\Delta u)^2\frac{\partial v}{\partial\nu}\,dS\\
 &=-\frac{1}{4}\int_{\partial\Omega}(\Delta u)^2\left\{\frac{\partial v}{\partial\nu}+\frac{(x-z)\cdot\nu}{n}\right\}\,dS. 
\end{align*}
Finally, the divergence theorem yields
\begin{equation*}
 \int_{\partial\Omega}\left\{\frac{\partial v}{\partial\nu}+\frac{(x-z)\cdot\nu}{n}\right\}\,dS=\int_\Omega\left(\Delta v+1\right)\,dx=0, 
\end{equation*}
and the proof is completed. 
\end{proof}

\section{Uniform pointwise estimate}
\label{section-uniform_estimates}

This section concerns the uniform estimates for solutions $u$ to the Dirichlet problem \eqref{dirichlet} for any domains $\Omega$ which are $\varepsilon_0$-close to $\mathbb{B}$ in the $C^{4+\alpha}$-sense. 
In particular, the closeness of $\Omega$ to $\mathbb{B}$ is used to obtain the boundary behavior
\begin{equation*}
 u(x)\geq \eta \left(d_{\partial\Omega}(x)\right)^2 \quad (x\in\Omega)
\end{equation*}
of solutions $u$ to \eqref{dirichlet}, where $d_{\partial\Omega}(x)$ denotes the distance from $x$ to $\partial\Omega$ and $\eta>0$ is a small constant. 
The constant $\varepsilon_0>0$ is chosen such that the following positivity preserving property holds. 

\begin{lemma}[Grunau and Robert \cite{grunau_robert-2010}]
\label{lemma-positivity_preserving}
There is a small constant $\varepsilon_0>0$ such that, for any $\Omega$ that is $\varepsilon_0$-close to $\mathbb{B}$ in the $C^{4+\alpha}$-sense, any $u\in C^{4+\alpha}(\overline{\Omega})$ satisfying
\begin{equation*}
 \left\{
 \begin{aligned}
  \Delta^2 u&\geq 0 && \text{in} \ \Omega,\\
  u=\frac{\partial u}{\partial\nu}&=0 && \text{on} \ \partial\Omega
 \end{aligned}
 \right.
\end{equation*}
must be nonnegative everywhere in $\Omega$. 
\end{lemma}

The proof is based on the blow-up analysis and the dependency of $\varepsilon_0$ on $n$ and $\alpha$ is not explicitly computable. 
Choosing $\varepsilon_0>0$ smaller if necessary, we also have the uniform Schauder estimates as follows. 

\begin{lemma}
\label{lemma-schauder}
There are constants $\varepsilon_0>0$ and $C>0$ such that solutions $u,\psi$ respectively to \eqref{dirichlet} and \eqref{torsion} satisfy
\begin{equation*}
 \|u\|_{C^{4+\alpha}(\overline{\Omega})}+\|\psi\|_{C^{4+\alpha}(\overline{\Omega})}\leq C
\end{equation*}
for any domain $\Omega$ that is $\varepsilon_0$-close to $\mathbb{B}$ in the $C^{4+\alpha}$-sense. 
\end{lemma}

\begin{proof}
Let us denote by $\Delta_\Omega^2\in\mathcal{L}(C_D^{4+\alpha}(\overline{\Omega}),C^\alpha(\overline{\Omega}))$ the biharmonic operator $\Delta^2$ acting on functions defined on $\Omega$, where
\begin{equation*}
 C_D^{4+\alpha}(\overline{\Omega}):=\left\{v\in C^{4+\alpha}(\overline{\Omega}) \;\middle|\; v=\frac{\partial v}{\partial\nu}=0 \ \text{on} \ \partial\Omega\right\}. 
\end{equation*}
By the Schauder theory (see Agmon, Douglis and Nirenberg \cite{agmon_douglis_nirenberg-1959}), we know the existence of the inverse $(\Delta^2_\Omega)^{-1}\in\mathcal{L}(C^\alpha(\overline{\Omega}),C_D^{4+\alpha}(\overline{\Omega}))$. 
Hence we need to show that the operator norms of $(\Delta^2_\Omega)^{-1}$ are estimated uniformly in $\Omega$. 

Recall that, by definition, there is a diffeomorphism $\Phi\in C^{4+\alpha}(\overline{\mathbb{B}},\overline{\Omega})$ with $\|\Phi-{\rm Id}\|_{C^{4+\alpha}(\overline{\mathbb{B}})}<\varepsilon_0$. 
Denoting by $\Phi_\ast$ and $\Phi^\ast$ respectively the push-forward and pull-back operators defined by
\begin{equation*}
 \Phi_\ast v(x):=v(\Phi^{-1}(x)), \quad \Phi^\ast u(\xi):=u(\Phi(\xi)) \quad (x\in\Omega, \ \xi\in\mathbb{B}), 
\end{equation*}
we readily check that
\begin{equation*}
 \|\Phi^\ast\Delta_\Omega^2\Phi_\ast-\Delta_\mathbb{B}^2\|_{\mathcal{L}(C_D^{4+\alpha}(\overline{\mathbb{B}}),C^\alpha(\overline{\mathbb{B}}))}<\left(2\|(\Delta_\mathbb{B}^2)^{-1}\|_{\mathcal{L}(C^{\alpha}(\overline{\mathbb{B}}),C_D^{4+\alpha}(\overline{\mathbb{B}}))}\right)^{-1}
\end{equation*}
holds for small $\varepsilon_0>0$. 
Hence, 
\begin{equation}
\label{neumann_series}
\begin{aligned}
 \Phi^\ast\left(\Delta_\Omega^2\right)^{-1}\Phi_\ast&=(\Phi^\ast\Delta_\Omega^2\Phi_\ast)^{-1}\\
 &=\left[I-\left(\Delta_\mathbb{B}^2\right)^{-1}\left(\Delta_\mathbb{B}^2-\Phi^\ast\Delta_\Omega^2\Phi_\ast\right)\right]^{-1}\left(\Delta_\mathbb{B}^2\right)^{-1}\\
 &=\sum_{k=0}^\infty \left[\left(\Delta_\mathbb{B}^2\right)^{-1}\left(\Delta_\mathbb{B}^2-\Phi^\ast\Delta_\Omega^2\Phi_\ast\right)\right]^k\left(\Delta_\mathbb{B}^2\right)^{-1}
\end{aligned}
\end{equation}
converges in $\mathcal{L}(C^{\alpha}(\overline{\mathbb{B}}),C_D^{4+\alpha}(\overline{\mathbb{B}}))$ and
\begin{align*}
 \left\|\left(\Delta_\Omega^2\right)^{-1}\right\|_{\mathcal{L}(C^{\alpha}(\overline{\Omega}),C_D^{4+\alpha}(\overline{\Omega}))}
 &\leq\|\Phi_\ast\|_{\mathcal{L}(C_D^{4+\alpha}(\overline{\mathbb{B}}),C_D^{4+\alpha}(\overline{\Omega}))}\|\Phi^\ast\|_{\mathcal{L}(C^{\alpha}(\overline{\Omega}),C^{\alpha}
 (\overline{\mathbb{B}}))}\\
 &\quad\times 2\left\|\left(\Delta_\mathbb{B}^2\right)^{-1}\right\|_{\mathcal{L}(C^\alpha(\overline{\mathbb{B}}),C_D^{4+\alpha}(\overline{\mathbb{B}}))}, 
\end{align*}
where the right hand side is bounded from above by a constant $C>0$ independent of $\Omega$. 
This yields
\begin{equation*}
 \|u\|_{C^{4+\alpha}(\overline{\Omega})}=\|(\Delta_\Omega^2)^{-1}[1]\|_{C^{4+\alpha}(\overline{\Omega})}\leq C. 
\end{equation*}
The estimate for $\psi$ follows in a similar manner. 
\end{proof}

\begin{proposition}
\label{proposition-distance}
Let $\varepsilon_0>0$ be a constant such that Lemmas \ref{lemma-positivity_preserving}, \ref{lemma-schauder} hold. 
Then there is a uniform constant $\eta>0$ such that
\begin{equation*}
 u(x)\geq \eta \left(d_{\partial\Omega}(x)\right)^2 \quad (x\in\Omega)
\end{equation*}
holds for a unique solution $u$ to \eqref{dirichlet} in any domain $\Omega$ that is $\varepsilon_0$-close to $\mathbb{B}$ in the $C^{4+\alpha}$-sense. 
\end{proposition}

\begin{proof}
Let $\psi\in C^{4+\alpha}(\overline{\Omega})$ be a solution to \eqref{torsion} and set
\begin{equation*}
 w:=u-c_\Omega\psi^2, \quad c_\Omega:=\left[2\left(1+2\max_{\overline{\Omega}}\left|D^2\psi\right|^2\right)\right]^{-1}\geq c_0>0, 
\end{equation*}
where the existence of the lower bound $c_0$ independent of $\Omega$ follows from Lemma \ref{lemma-schauder}. 
Then $w$ satisfies $w=\partial_\nu w=0$ on $\partial\Omega$ and
\begin{align*}
 \Delta^2 w&=\Delta\left\{\Delta u+2c_\Omega\left(\psi-|\nabla\psi|^2\right)\right\}\\
 &=1-2c_\Omega\left(1+2\left|D^2\psi\right|^2\right)\geq 0 \quad \text{in} \ \Omega. 
\end{align*}
By Lemma \ref{lemma-positivity_preserving}, we have $w\geq 0$, i.e., 
\begin{equation*}
 u\geq c_\Omega\psi^2\geq c_0\mu_0^2\left(d_{\partial\Omega}(x)\right)^2 \quad \text{in} \ \Omega, 
\end{equation*}
where $\psi\geq \mu_0d_{\partial\Omega}(x)$ for a small uniform constant $\mu_0>0$ follows from (a proof of) Hopf's lemma with the uniform interior sphere condition of $\Omega$. 
\end{proof}

A glimpse of the proof of Lemma \ref{lemma-schauder} also yields the following perturbation result, which will be used in the next section. 
Note that 
\begin{equation}
\label{explicit_radial_solution}
 u_0=\frac{(|x|^2-1)^2}{8n(n+2)} \quad (x\in\mathbb{R}^n)
\end{equation}
is a unique solution to \eqref{dirichlet} for $\Omega=\mathbb{B}$ and $-\Delta u_0$ is a quadratic polynomial of the form \eqref{quadratic_polynomial}. 

\begin{lemma}
\label{lemma-perturbation}
For any small constant $c>0$, there is $\varepsilon_0>0$ such that 
\begin{equation*}
 \left\|u-u_0\right\|_{C^{4+\alpha}(\overline{\Omega})}<c
\end{equation*}
for any domain $\Omega $ that is $\varepsilon_0$-close to $\mathbb{B}$ in the $C^{4+\alpha}$-sense, where $u\in C^{4+\alpha}(\overline{\Omega})$ is a solution to \eqref{dirichlet}. 
In particular, $v:=-\Delta u$ attains its maximum at an interior point $z\in\Omega$ with $|z|<1/2$. 
\end{lemma}

\begin{proof}
In view of \eqref{neumann_series}, we see that
\begin{equation*}
 \|\Phi^\ast u-u_0\|_{C^{4+\alpha}(\overline{\mathbb{B}})}=\|\Phi^\ast(\Delta_\Omega^2)^{-1}\Phi_\ast[1]-(\Delta_\mathbb{B}^2)^{-1}[1]\|_{C^{4+\alpha}(\overline{\mathbb{B}})}
\end{equation*}
can be arbitrarily small by choosing a sufficiently small $\varepsilon_0>0$. 
We also have $\|\Phi_\ast u_0-u_0\|_{C^{4+\alpha}(\overline{\Omega})}<c/2$ for small $\varepsilon_0>0$. 
\end{proof}

\section{Weighted inequalities for harmonic functions}
\label{section-weighted_estimates}

In order to derive the stability estimate \eqref{stability_estimate}, we shall take the approach by Magnanini and Poggesi \cite{magnanini_poggesi-2020} with our integral identity \eqref{integral_identity_4th}. 
Let us first observe that, by using the harmonic function
\begin{equation}
\label{harmonic_difference}
 h:=v-Q=-\Delta u-\frac{R^2-|x-z|^2}{2n}, 
\end{equation}
the integral identity \eqref{integral_identity_4th} can be written as
\begin{equation}
\label{harmonic_identity}
 \int_\Omega u|D^2 h|^2\,dx=\frac14\int_{\partial\Omega}\left\{c^2-(\Delta u)^2\right\}\frac{\partial h}{\partial\nu}\,dS, 
\end{equation}
and $\rho_2-\rho_1$ can be related to the oscillation of $h$ on $\partial\Omega$ by
\begin{equation}
\label{distance_to_oscillation}
\begin{aligned}
 {\rho_2}^2-{\rho_1}^2&=\max_{x\in\partial\Omega}|x-z|^2-\min_{x\in\partial\Omega}|x-z|^2\\
 &\leq 2n\left(\max_{\partial\Omega}h-\min_{\partial\Omega}h\right)+4n\|\Delta u-c\|_{L^\infty(\partial\Omega)}. 
\end{aligned}
\end{equation}
By Lemma \ref{lemma-perturbation}, we may assume that $z\in\Omega$ with $|z|<1/2$ is a maximum point of $v$ in $\Omega$ and thus
\begin{equation*}
 \nabla h(z)=\nabla v(z)-\nabla Q(z)=0, 
\end{equation*}
and moreover $\|\nabla h\|_{L^\infty(\Omega)}+\|D^2h\|_{L^\infty(\Omega)}$ is sufficiently small. 

In view of Proposition \ref{proposition-distance}, we can relate the oscillation of $h$ to the left hand side of \eqref{harmonic_identity} by the chain of inequalities
\begin{equation}
\label{sequence_inequalities}
\begin{aligned}
 \max_{\partial\Omega}h-\min_{\partial\Omega}h
 &\leq C_1\|h-h_\Omega\|_{L^{2^\ast}(\Omega)}^{2^\ast/(n+2^\ast)}\\
 &\leq C_2\|\nabla h\|_{L^2(\Omega)}^{2^\ast/(n+2^\ast)}
 \leq C_3\|d_{\partial\Omega}D^2h\|_{L^2(\Omega)}^{2^\ast/(n+2^\ast)}, 
\end{aligned}
\end{equation}
where $h_\Omega:=|\Omega|^{-1}\int_\Omega h\,dx$ is the mean value of $h$ over $\Omega$, $2^\ast:=2n/(n-2)$ for $n\geq3$ and $2^\ast$ is arbitrarily large number for $n=2$. 
The first inequality is due to Magnanini and Poggesi \cite{magnanini_poggesi-2020}, and the second one is the Poincar\'e-Sobolev inequality, and the third one was essentially proved by Hurri-Syrj\"anen \cite{hurri-syrjanen-1994}. 
We emphasize that the only nonlinear inequality is the first one and this reflects in the nonlinear nature of the stability estimate \eqref{stability_estimate}. 

\begin{lemma}[Hurri-Syrj\"anen \cite{hurri-syrjanen-1994}, Magnanini and Poggesi {\cite{magnanini_poggesi-2020}}]
\label{lemma-harmonic_inequality}
There are $\varepsilon_0>0$ and $C>0$ such that, if $\Omega$ is $\varepsilon_0$-close to $\mathbb{B}$ in the $C^{4+\alpha}$-sense, then the harmonic function $h$ defined by \eqref{harmonic_difference} satisfies
\begin{align*}
 \max_{\partial\Omega} h-\min_{\partial\Omega} h&\leq C\|h-h_\Omega\|_{L^{2^\ast}(\Omega)}^{2^\ast/(n+2^\ast)},\\
 \int_\Omega|\nabla h|^2\,dx&\leq C\int_\Omega d_{\partial\Omega}(x)^2|D^2 h|^2\,dx. 
\end{align*}
\end{lemma}

\begin{proof}
See \cite[Corollary 2.3]{magnanini_poggesi-2020} for the proof of the second inequality. 
In \cite[Lemma 2.6]{magnanini_poggesi-2020}, the first inequality was proved when $\|\nabla h\|_{L^\infty(\Omega)}$ is uniformly bounded and $\|h-h_\Omega\|_{L^{2^\ast}(\Omega)}$ is sufficiently small. 
This assumption is fulfilled by Lemmas \ref{lemma-schauder} and \ref{lemma-perturbation}. 
\end{proof}

\begin{remark}
\label{remark-harmonic_inequalities}
{\rm
The inequalities in Lemma \ref{lemma-harmonic_inequality} are valid in more general forms (see \cite[Theorem 2.3]{magnanini_poggesi-2021} and \cite[Lemma 2.1]{magnanini_poggesi-2020}). 
In particular, the former inequality is still valid without assuming the smallness of $\|h-h_\Omega\|_{L^{2^\ast}(\Omega)}$, and thus the use of Lammas \ref{lemma-schauder} and \ref{lemma-perturbation} is not necessary. 
As for the latter inequality, it was proved that
\begin{equation*}
 \left(\int_\Omega|v-v_{\Omega}|^r\,dx\right)^{1/r}\leq C\left(\int_\Omega \left(d_{\partial\Omega}(x)^{\alpha}|\nabla v|\right)^p\,dx\right)^{1/p}
\end{equation*}
holds for any (even non-harmonic) $v$ if $1\leq p\leq r\leq np\{n-p(1-\alpha)\}^{-1}<\infty$ and $0\leq\alpha\leq 1$. 
A simple scaling argument shows that, at least for $\alpha=1$, the range of the exponents, i.e., $1\leq p=r<\infty$, is optimal. 
Moreover, $\alpha$ cannot exceed $1$, as we can see by inserting the inverse power of a distance-like function $v\sim {d_{\partial\Omega}}^{-1}$. 
This will be an obstacle in extending our results to polyharmonic operators $(-\Delta)^m$ with $m\geq 3$. 
}
\end{remark}

Since $\|\partial_\nu h\|_{L^\infty(\partial\Omega)}$, $\|\Delta u+c\|_{L^\infty(\partial\Omega)}$ and $\rho_2+\rho_1$ are uniformly bounded with respect to $\Omega$ by Lemma \ref{lemma-schauder}, the combination of \eqref{harmonic_identity}, \eqref{distance_to_oscillation} and \eqref{sequence_inequalities} immediately results in the weaker stability estimate
\begin{equation*}
 \rho_2-\rho_1\leq C\left(\left\|\Delta u-c\right\|_{L^1(\partial\Omega)}^{1/n}+\left\|\Delta u-c\right\|_{L^\infty(\partial\Omega)}\right). 
\end{equation*}

We shall improve this estimate by carefully treat $\partial_\nu h$ in the right hand side of \eqref{harmonic_identity}. 
In fact, for the second order problem \eqref{torsion} and \eqref{od-2}, Feldman \cite{feldman-2018} derived the {\it linear} weighted trace inequality
\begin{equation*}
 \int_{\partial\Omega}|\nabla h|^2\,dS\leq C\int_\Omega d_{\partial\Omega}(x) |D^2 h|^2\,dx, 
\end{equation*}
which combined with \eqref{integral_identity_2nd} yields
\begin{equation*}
 \|\nabla h\|_{L^2(\partial\Omega)}\leq C\||\nabla\psi|-c\|_{L^2(\partial\Omega)}, 
\end{equation*}
and this estimate was used to obtain the stability estimate \eqref{sharp_estimate} (see Magnanini and Poggesi \cite{magnanini_poggesi-2020}). 
However, for \eqref{harmonic_identity}, we need to have such a trace inequality with quadratic weight ${d_{\partial\Omega}}^2$ and this type of linear estimate cannot hold in general even if the norm of the left hand side is weakened, say, to the $L^1$-norm, as one can check by inserting harmonic polynomials of higher degree to $h$. 

Our improvement of the stability estimate relies on the following {\it nonlinear} weighted trace inequality for small harmonic functions, inspired by Lemma \ref{lemma-harmonic_inequality}, where
\begin{equation*}
 \beta_p:=\left\{
 \begin{aligned}
  & \hspace{7mm} \frac13 && (1\leq p\leq 3),\\
  & \frac{n+p-1}{(n+2)p} && (3<p<\infty),\\
  & \hspace{3mm} \frac{1}{n+2} && (p=\infty). 
 \end{aligned}
 \right.
\end{equation*}

\begin{lemma}
\label{lemma-nonlinear_trace_inequality}
For $1\leq p\leq\infty$ and $0\leq\beta<\beta_p$, there are $\varepsilon_0>0$ and $C>0$ such that, if $\Omega$ is $\varepsilon_0$-close to $\mathbb{B}$ in the $C^{4+\alpha}$-sense, then the harmonic function $h$ defined by \eqref{harmonic_difference} satisfies
\begin{equation}
\label{nonlinear_trace_inequality}
 \left(\int_{\partial\Omega}\left|\nabla h\right|^p\,dS\right)^{1/p}\leq C\left[\int_\Omega d_{\partial\Omega}(x)^2|D^2h|^2\,dx\right]^{\beta}. 
\end{equation}
Moreover, for $p=\infty$, the above estimate holds with $\beta=\beta_\infty=1/(n+2)$. 
\end{lemma}

\begin{proof}
We denote by $C>0$ a generic constant independent of $\Omega$ that may change. 
For each fixed $x\in\partial\Omega$, set $\nu=\nu(x)$ and take a unit vector $e\in\mathbb{R}^n$ such that $|\nabla h(x)|=e\cdot\nabla h(x)$. 
By the mean value property of the harmonic function $y\mapsto e\cdot\nabla h(y)$ followed by H\"older's inequality, 
\begin{align*}
 \left|\nabla h(x)\right|&=\frac{1}{\omega_n s^n}\int_{B_s(x-s\nu)}e\cdot\nabla h(y)\,dy+\int_0^se\cdot D^2h(x-t\nu)\nu\,dt\\
 &\leq \frac{\|\nabla h\|_{L^2(\Omega)}}{\omega_n^{1/2}s^{n/2}}+\left[\int_0^s t^2|D^2h|^2\,dt\right]^{\frac{\theta}{2}}\left[\int_0^s t^{-\frac{2\theta}{2-\theta}}|D^2h|^\frac{2(1-\theta)}{2-\theta}\,dt\right]^{\frac{2-\theta}{2}}\\
 &\leq \frac{\|\nabla h\|_{L^2(\Omega)}}{\omega_n^{1/2}s^{n/2}}+\left[\int_0^s t^2|D^2h|^2\,dt\right]^{\frac{\theta}{2}}\|D^2h\|_{L^\infty(\Omega)}^{1-\theta}\left(\frac{2-\theta}{2-3\theta}\right)^\frac{2-\theta}{2}s^{\frac{2-3\theta}{2}}
\end{align*}
holds for any small $s>0$, $0\leq\theta<2/3$, $\omega_n:=|\mathbb{B}|$ and $D^2h:=D^2h(x-t\nu)$. 
Since $\|D^2h\|_{L^\infty(\Omega)}$ is uniformly bounded by Lemma \ref{lemma-schauder}, integrating over $\partial\Omega$, we have
\begin{equation}
\label{l1_middle_estimate}
 \int_{\partial\Omega}\left|\nabla h\right|\,dS
 \leq C\left(\frac{\|\nabla h\|_{L^2(\Omega)}}{s^{n/2}}+\left[\int_\Omega d_{\partial\Omega}(x)^2|D^2h|^2\,dx\right]^{\frac{\theta}{2}}s^{\frac{2-3\theta}{2}}\right). 
\end{equation}
In order to optimize the inequality, we choose
\begin{equation*}
 s=s_\ast:=\left(\frac{n\|\nabla h\|_{L^2(\Omega)}}{(2-3\theta)\left[\int_\Omega d_{\partial\Omega}(x)^2|D^2h|^2\,dx\right]^{\theta/2}}\right)^{\frac{2}{n+2-3\theta}}, 
\end{equation*}
which together with Lemma \ref{lemma-harmonic_inequality} results in the desired $L^1$-estimate
\begin{align*}
 \int_{\partial\Omega}\left|\nabla h\right|\,dS
 &\leq C\|\nabla h\|_{L^2(\Omega)}^\frac{2-3\theta}{n+2-3\theta}\left[\int_\Omega d_{\partial\Omega}(x)^2|D^2h|^2\,dx\right]^{\frac{n\theta}{2(n+2-3\theta)}}\\
 &\leq C\left[\int_\Omega d_{\partial\Omega}(x)^2|D^2h|^2\,dx\right]^{\frac{n\theta+2-3\theta}{2(n+2-3\theta)}}, 
\end{align*}
where the exponent approaches $1/3$ as $\theta\to 2/3$. 
Note that $s_\ast>0$ is indeed admissible, since Lemma \ref{lemma-harmonic_inequality} yields
\begin{equation*}
 s_\ast\leq C\left[\int_\Omega d_{\partial\Omega}(x)^2|D^2h|^2\,dx\right]^\frac{1-\theta}{n+2-3\theta}, 
\end{equation*}
and the right hand side can be arbitrarily small for fixed $0\leq\theta<2/3$ if $\varepsilon_0>0$ is chosen small enough by Lemma \ref{lemma-perturbation}. 

For general $1\leq p\leq \infty$, we observe that the estimate \eqref{l1_middle_estimate} with the left hand side replaced by $\|\nabla h\|_{L^p(\partial\Omega)}$ holds with the additional requirement $\theta\leq 2/p$. 
Thus \eqref{nonlinear_trace_inequality} holds for $1\leq p\leq 3$. 
For $3<p\leq\infty$, we choose $\theta=2/p$ to obtain
\begin{equation*}
 \left\|\nabla h\right\|_{L^p(\partial\Omega)}\leq C\left[\int_\Omega d_{\partial\Omega}(x)^2|D^2h|^2\,dx\right]^{\frac{n+p-3}{np+2p-6}}, 
\end{equation*}
where the exponent is understood to be $1/(n+2)$ for $p=\infty$. 
We can improve the estimate more when $3<p<\infty$ by interpolating the special case of this inequality for $p=\infty$ and \eqref{nonlinear_trace_inequality} for $p=3$ and $\beta<1/3$ as
\begin{align*}
 \left\|\nabla h\right\|_{L^p(\partial\Omega)}
 &\leq \left\|\nabla h\right\|_{L^\infty(\partial\Omega)}^{1-\frac{3}{p}}\left\|\nabla h\right\|_{L^3(\partial\Omega)}^{\frac{3}{p}}\\
 &\leq C\left[\int_\Omega d_{\partial\Omega}(x)^2|D^2h|^2\,dx\right]^{\frac{p-3+3(n+2)\beta}{(n+2)p}}, 
\end{align*}
where the exponent approaches $(n+p-1)/(n+2)p$ as $\beta\to 1/3$. 
\end{proof}

We now combine all the ingredients to prove Theorem \ref{theorem-stability}. 

\begin{proof}[Proof of Theorem \ref{theorem-stability}]
Let $C>0$ denote a generic constant independent of $\Omega$. 
By Proposition \ref{proposition-distance}, the integral identity \eqref{harmonic_identity} and Lemma \ref{lemma-nonlinear_trace_inequality}, 
\begin{align*}
 \left\|\nabla h\right\|_{L^p(\partial\Omega)}
 &\leq C\left[\int_{\Omega}d_{\partial\Omega}(x)^2\left|D^2h\right|^2\,dx\right]^\beta\\
 &\leq C\left[\int_{\partial\Omega}\left|\Delta u-c\right|\left|\frac{\partial h}{\partial\nu}\right|\,dS\right]^\beta\\
 &\leq C\left\|\Delta u-c\right\|_{L^{p'}(\partial\Omega)}^\beta\left\|\nabla h\right\|_{L^p(\partial\Omega)}^\beta, 
\end{align*}
where $1\leq p,p'\leq\infty$ are chosen to satisfy $(1/p)+(1/p')=1$. 
Hence, 
\begin{equation*}
 \left\|\nabla h\right\|_{L^p(\partial\Omega)}\leq C\left\|\Delta u-c\right\|_{L^{p'}(\partial\Omega)}^\frac{\beta}{1-\beta}. 
\end{equation*}
Combining the above estimate with \eqref{harmonic_identity}, \eqref{distance_to_oscillation} and \eqref{sequence_inequalities}, we obtain
\begin{align*}
 \rho_2-\rho_1&\leq C\left(\left[\int_{\partial\Omega}\left|\Delta u-c\right|\left|\frac{\partial h}{\partial\nu}\right|\,dS\right]^\frac{2^\ast}{2(n+2^\ast)}+\left\|\Delta u-c\right\|_{L^\infty(\partial\Omega)}\right)\\
 &\leq C\left(\left[\left\|\Delta u-c\right\|_{L^{p'}(\partial\Omega)}\left\|\nabla h\right\|_{L^p(\partial\Omega)}\right]^\frac1n+\left\|\Delta u-c\right\|_{L^\infty(\partial\Omega)}\right)\\
 &\leq C\left(\left\|\Delta u-c\right\|_{L^{p'}(\partial\Omega)}^\frac{1}{n(1-\beta)}+\left\|\Delta u-c\right\|_{L^\infty(\partial\Omega)}\right)
\end{align*}
for $n\geq 3$. 
If $n=2$, the same estimate follows with the exponent $1/n$ replaced by any positive number smaller than $1/2$. 
By exchanging the roles of $p$ and $p'$, we conclude the proof. 
\end{proof}

\bigskip

\noindent
{\bf Acknowledgments.}
The second author was supported in part by the Grant-in-Aid for Scientific Research (C) 20K03673, Japan Society for the Promotion of Science. 



\begin{thebibliography}{99}
\bibitem{aftalion_busca_reichel-1999}
Aftalion, A.; Busca, J.; Reichel, W., 
Approximate radial symmetry for overdetermined boundary value problems. 
{\it Adv.\ Differential Equations} {\bf 4} (1999), {\it no.~6}, 907--932. 
\bibitem{agmon_douglis_nirenberg-1959} Agmon, S.; Douglis, A.; Nirenberg, L. 
Estimates near the boundary for solutions of elliptic partial differential equations satisfying general boundary conditions.\ I. 
{\it Comm.\ Pure Appl.\ Math.} {\bf 12} (1959) 623--727. 
\bibitem{barkatou-2008}
Barkatou, M., 
A symmetry result for a fourth order overdetermined boundary value problem. 
{\it Appl.\ Math.\ E-Notes} {\bf 8} (2008), 76--81. 
\bibitem{bennett-1986}
Bennett, A., 
Symmetry in an overdetermined fourth order elliptic boundary value problem. 
{\it SIAM J. Math.\ Anal. } {\bf 17} (1986), {\it no.~6},1354--1358. 
\bibitem{brandolini_nitsch_salani_trombetti-2008_2}
Brandolini, B.; Nitsch, C.; Salani, P.; Trombetti, C., 
On the stability of the Serrin problem. 
{\it J. Differential Equations} {\bf 245} (2008), {\it no.~6}, 1566--1583. 
\bibitem{ciraolo_magnanini_vespri-2016}
Ciraolo, G; Magnanini, R.; Vespri, V., 
H\"older stability for Serrin's overdetermined problem. 
{\it Ann.\ Mat.\ Pura Appl.} {\bf 195} (2016), {\it no.~4}, 1333--1345. 
\bibitem{colasuonno_vecchi-2019}
Colasuonno, F.; Vecchi, E., 
Symmetry and rigidity for the hinged composite plate problem. 
{\it J. Differential Equations} {\bf 266} (2019), {\it no.~8}, 4901--4924. 
\bibitem{dalmasso-1990}
Dalmasso, R., 
Un probl\`eme de sym\'etrie pour une \'equation biharmonique. 
{\it Ann.\ Fac.\ Sci.\ Toulouse Math.\ (5)} {\bf 11} (1990), {\it no.~3}, 45--53. 
\bibitem{dipierro_poggesi_valdinoci-2021}
Dipierro, S.; Poggesi, G.; Valdinoci, E., 
A Serrin-type problem with partial knowledge of the domain. 
{\it Nonlinear Anal.} {\bf 208} (2021), 112330, 44 pp. 
\bibitem{feldman-2018}
Feldman, W. M., 
Stability of Serrin's problem and dynamic stability of a model for contact angle motion. 
{\it SIAM J. Math.\ Anal.} {\bf 50} (2018), {\it no.~3} 3303--3326. 
\bibitem{fogagnolo_pinamonti-2021}
Fogagnolo, M.; Pinamonti, A., 
New integral estimates in substatic Riemannian Manifolds and the Alexandrov theorem
Preprint is available at {\it https://arxiv.org/abs/2105.04672}. 
\bibitem{gazzola_grunau_sweers-2010}
Gazzola, F.; Grunau, H.-Ch.; Sweers, G., 
Polyharmonic boundary value problems. 
Positivity preserving and nonlinear higher order elliptic equations in bounded domains. 
Lecture Notes in Mathematics, 1991. {\it Springer-Verlag, Berlin}, 2010. 
\bibitem{gilsbach_onodera-2021}
Gilsbach, A.; Onodera, M., 
Linear stability estimates for Serrin's problem via a modified implicit function theorem. 
Preprint is available at {\it https://arxiv.org/abs/2103.07072}. 
\bibitem{gilsbach_stollenwerk-2021}
Gilsbach, A.; Stollenwerk, K., 
Existence and stability of solutions for a fourth order overdetermined problem. 
{\it in preparation}. 
\bibitem{grunau_robert-2010}
Grunau, H.-Ch.; Robert, F., 
Positivity and almost positivity of biharmonic Green's functions under Dirichlet boundary conditions. 
{\it Arch.\ Ration.\ Mech.\ Anal.} {\bf 195} (2010), {\it no.~3}, 865--898. 
\bibitem{hurri-syrjanen-1994}
Hurri-Syrj\"anen, R., 
An improved Poincar\'e inequality. 
{\it Proc.\ Amer.\ Math.\ Soc.} {\bf 120} (1994), {\it no.~1}, 213--222. 
\bibitem{magnanini_poggesi-2019}
Magnanini, R.; Poggesi, G., 
On the stability for Alexandrov's soap bubble theorem. 
{\it J. Anal.\ Math.} {\bf 139} (2019), {\it no.~1}, 179--205. 
\bibitem{magnanini_poggesi-2020_indiana}
Magnanini, R.; Poggesi, G., 
Serrin's problem and Alexandrov's soap bubble theorem: enhanced stability via integral identities. 
{\it Indiana Univ.\ Math.\ J.} {\bf 69} (2020), {\it no.~4}, 1181--1205. 
\bibitem{magnanini_poggesi-2020}
Magnanini, R.; Poggesi, G., 
Nearly optimal stability for Serrin’s problem and the Soap Bubble theorem. 
{\it Calc.\ Var.\ Partial Differential Equations} {\bf 59} (2020), {\it no.~1, Paper No.~35, 23 pp.}
\bibitem{magnanini_poggesi-2021}
Magnanini, R.; Poggesi, G., 
An interpolating inequality for solutions of uniformly elliptic equations. 
{\it Geometric properties for parabolic and elliptic PDE's}, 233--245, {\it Springer INdAM Series}, {\bf 47}, Springer, 2021. 
\bibitem{philippin_ragoub-1995}
Philippin, G. A.; Ragoub, L., 
On some second order and fourth order elliptic overdetermined problems. 
{\it Z. Angew.\ Math.\ Phys.} {\bf 46} (1995), {\it no.~2}, 188--197. 
\bibitem{payne_schaefer-1989}
Payne, L. E.; Schaefer, P. W., 
Duality theorems in some overdetermined boundary value problems. 
{\it Math.\ Methods Appl.\ Sci.} {\bf 11} (1989), {\it no.~6}, 805--819. 
\bibitem{payne_schaefer-1994}
Payne, L. E.; Schaefer, P. W., 
On overdetermined boundary value problems for the biharmonic operator. 
{\it J. Math.\ Anal.\ Appl.} {\bf 187} (1994), {\it no.~2}, 598--616. 
\bibitem{pucci_serrin-1986}
Pucci, P.; Serrin, J., 
A general variational identity. 
{\it Indiana Univ.\ Math.\ J.} {\bf 35} (1986), {\it no.~3}, 681--703. 
\bibitem{scheuer-2021}
Scheuer, J., 
Stability from rigidity via umbilicity. 
Preprint is available at {\it https://arxiv.org/abs/2103.07178}. 
\bibitem{serrin-1971}
Serrin, J., 
A symmetry problem in potential theory. 
{\it Arch.\ Ration.\ Mech.\ Anal.}\ {\bf 43} (1971), 304--318. 
\bibitem{weinberger-1971}
Weinberger, H. F., 
Remark on the preceding paper of Serrin. 
{\it Arch.\ Ration.\ Mech.\ Anal.}\ {\bf 43} (1971), 319--320. 
\end{thebibliography}
\end{document}